\documentclass[10pt,reqno]{amsart}

\usepackage{amssymb}
\usepackage{amsmath}
\usepackage{amsfonts}
\usepackage{amscd}
\usepackage{mathrsfs}
\usepackage{color}
\usepackage{bigints}

\newtheorem{thm}{Theorem}[section]

\newtheorem{lem}[thm]{Lemma}
\newtheorem{prop}[thm]{Proposition}

\theoremstyle{definition}

\theoremstyle{remark}

\theoremstyle{remark}

\numberwithin{equation}{section}

\newcommand{\bdy}{\partial\mathbf{\Omega}}
\newcommand{\BR}{\mathbf{B}_R}
\newcommand{\BRc}{\mathbf{B}^c_R}
\newcommand{\Jc}{\mathcal{J}_\lambda}
\newcommand{\Jcll}{\mathcal{J}_{\lambda(l)}}
\newcommand{\E}{\mathscr{E}}
\newcommand{\Ms}{\mathscr{M}}
\newcommand{\Om}{\mathbf{\Omega}}
\newcommand{\Ov}{\Om_\varepsilon}
\newcommand{\RN}{\mathbf{R}^N}
\newcommand{\R}{\mathbf{R}}
\newcommand{\La}{-\Delta\hspace{0.2mm}}
\newcommand{\Lp}{-\Delta_p\hspace{0.2mm}}

\newcommand{\bre}[1]{\left\{#1\right\}}
\newcommand{\pr}[1]{\left(#1\right)}
\newcommand{\n}[1]{\left\vert#1\right\vert}
\newcommand{\nm}[1]{\left\Vert#1\right\Vert}

\newcommand{\dx}{\hspace{0.2mm}dx}
\newcommand{\gradu}{\nabla u}
\newcommand{\ul}{u_\lambda}
\newcommand{\gradul}{\nabla\ul}
\newcommand{\ull}{u_{\lambda(l)}}
\newcommand{\gradull}{\nabla\ull}
\newcommand{\um}{u_\mu}
\newcommand{\gradum}{\nabla\um}
\newcommand{\gradv}{\nabla v}
\newcommand{\vl}{v_\lambda}
\newcommand{\gradvl}{\nabla\vl}
\newcommand{\gradw}{\nabla\omega}
\newcommand{\IOm}{\int_{\Om}}
\newcommand{\IOv}{\int_{\Ov}}
\newcommand{\IR}{\int_{\RN}}
\newcommand{\IBR}{\int_{\mathbf{B}_R}}
\newcommand{\IBRc}{\int_{\mathbf{B}^c_R}}
\newcommand{\ult}{\tilde{u}_\lambda}
\newcommand{\ve}{v_\varepsilon}
\newcommand{\vv}{\varphi_\varepsilon}

\newcommand{\CCON}{C^1_c\big(\RN\big)}
\newcommand{\LalocR}{L^\alpha_{loc}\big(\RN\big)}
\newcommand{\LsR}{L^\sigma\big(\RN\big)}
\newcommand{\LlocR}{L^1_{loc}\big(\RN\big)}
\newcommand{\WpR}{W^{1,\hspace{0.2mm}p}\big(\RN\big)}
\newcommand{\WpRnm}{W^{1,\hspace{0.2mm}p}(\RN)}
\newcommand{\WpOvnm}{W^{1,\hspace{0.2mm}p}(\Ov)}
\newcommand{\LqhR}{L^q_h\big(\RN\big)}
\newcommand{\LqhRnm}{L^q_h(\RN)}
\newcommand{\LqhOvnm}{L^q_h(\Ov)}
\newcommand{\LrkR}{L^r_k\big(\RN\big)}
\newcommand{\LrkRnm}{L^r_k(\RN)}
\newcommand{\LrkOvnm}{L^r_k(\Ov)}
\begin{document}

\title[]
{\bf Elliptic variational problems with mixed nonlinearities}

\author[]
{\sf Qi Han}

\address{Department of Science and Mathematics, Texas A\&M University at San Antonio
\vskip 2pt San Antonio, Texas 78224, USA \hspace{14.4mm}{\sf Email: qhan@tamusa.edu}}

%\email{}
%\date{}
%\commby{}

% -----------------------------------------------------------------------------

\begin{abstract}
In this paper, we study the existence and multiplicity results of nontrivial positive solutions to a quasilinear elliptic equation in $\RN$, when $N\geq2$, as
\begin{equation}
\Lp u+u^{p-1}=\lambda\hspace{0.2mm}k(x)u^{r-1}-h(x)u^{q-1}.\nonumber
\end{equation}
Here, $h(x),k(x)>0$ are Lebesgue measurable functions, $1<p<q<\infty$, $p<r<\min\{p^*,q\}$ if $p<N$ while $p<r<q$ if $p\geq N$, and $\lambda>0$ is a parameter.
\end{abstract}

% -----------------------------------------------------------------------------
\maketitle
% -----------------------------------------------------------------------------

% +++ below is section 1 %%%%%%%%%%%%%%%%%%%%%%%%%%%%%%%%%%%%%%%%%%%%%%%%%%%%%%%%%%%%%%%%%%%%%%%%%%%%%%%%%%%%%%%%%%%%%%%%%%%%%%%%%%%%%%%%%%%%%%%%%%%
% &&& %%%%%%%%%%%%%%%%%%%%%%%%%%%%%%%%%%%%%%%%%%%%%%%%%%%%%%%%%%%%%%%%%%%%%%%%%%%%%%%%%%%%%%%%%%%%%%%%%%%%%%%%%%%%%%%%%%%%%%%%%%%%%%%%%%%%%%%%%%%%%%
\section{Introduction}\label{Int} % use lowercase except for proper names
\noindent In 1996, Alama and Tarantello \cite{AT} investigated an elliptic problem, when  $N\geq3$, as
\begin{equation}\label{eq1.1}
\La u=\lambda\hspace{0.2mm}u+k(x)u^{r-1}-h(x)u^{q-1},\hspace{2mm}u>0
\end{equation}
on a bounded, smooth domain $\Om$, subject to $u=0$ on $\bdy$, for $h(x),k(x)>0$, $2<r<q$ and a constant $\lambda\in\R$, where the existence and multiplicity results of solutions to \eqref{eq1.1} are discussed assuming either $\bigintsss_\Om\big[\frac{k^q(x)}{h^r(x)}\big]^{\frac{1}{q-r}}\dx<\infty$ or $\bigintsss_\Om\big[\frac{k^{q-2}(x)}{h^{r-2}(x)}\big]^{\frac{2^*}{(2^*-2)(q-r)}}\dx<\infty$.
Notice that the latter condition is weaker than the former condition when $r<2^*$ and $\IOm k^{\frac{2^*}{2^*-r}}(x)\dx<\infty$.

It seems to the author that Chabrowski \cite{Ch} is the first person who studied a similar problem to \eqref{eq1.1} in $\RN$, with
$\Lp u=-\operatorname{div}\big(\n{\gradu}^{p-2}\gradu\big)$ and $p<r<q<p^*$, as
\begin{equation}\label{eq1.2}
\Lp u+u^{p-1}=\lambda\hspace{0.2mm}u^{r-1}-h(x)u^{q-1},\hspace{2mm}u\geq0,
\end{equation}
where the existence result of solutions to \eqref{eq1.2} is discussed assuming that $\IR h^{\frac{r}{r-q}}(x)\dx<\infty$; some time later, Pucci and R\u{a}dulescu \cite{PR,Ra} extended the above work of Chabrowski, and first studied the existence and multiplicity results of solutions to \eqref{eq1.2} under the same assumption.
Quite recently, Autuori and Pucci \cite{AP1,AP2}, Pucci and Zhang \cite{PZ}, and R\u{a}dulescu \textsl{et al.} \cite{RXZ} further studied an extension of \eqref{eq1.2} in $\RN$, for $p<r<\min\{p^*,q\}$, as
\begin{equation}\label{eq1.3}
\Lp u+w(x)u^{p-1}=\lambda\hspace{0.2mm}k(x)u^{r-1}-h(x)u^{q-1},\hspace{2mm}u\geq0,
\end{equation}
with $w(x)\simeq\frac{1}{(1+\n{x})^p}>0$\footnote{In \cite{RXZ}, the condition $w(x)\geq w_0>0$ is used.} and $p<N$, where the existence and multiplicity results of solutions to \eqref{eq1.3} are shown using either the conditions in \cite{AT} or $\bigintsss_\RN\big[\frac{k^{q-1}(x)}{h^{r-1}(x)}\big]^{\frac{p^*}{(p^*-1)(q-r)}}\dx<\infty$.
Again, the latter is weaker than $\bigintsss_\RN\big[\frac{k^q(x)}{h^r(x)}\big]^{\frac{1}{q-r}}\dx<\infty$ if $r<p^*$ and $\IR k^{\frac{p^*}{p^*-r}}(x)\dx<\infty$.

In this paper, for $N\geq2$, we study the existence of nontrivial solutions to
\begin{equation}\label{eq1.4}
\Lp u+u^{p-1}=\lambda\hspace{0.2mm}k(x)u^{r-1}-h(x)u^{q-1},\hspace{2mm}u\geq0
\end{equation}
in $\RN$, where $1<p<q<\infty$, $p<r<\min\{p^*,q\}$ with $p^*=\frac{Np}{N-p}$ if $p<N$ while $p<r<q$ if $p\geq N$, $h(x),k(x)>0$ are Lebesgue measurable functions, and $\lambda>0$ is a constant.

Our primary assumptions are as follows.

\vskip 2pt
{\bf Standing Assumptions.}
\vskip 2pt
\noindent{\bf(1-a).} When $1<p<N$, $k(x)\in\LalocR$ for some $\alpha\in\big(\frac{p^*}{p^*-r},\infty\big]$;
\vskip 2pt
\noindent{\bf(1-b).} when $N\leq p<\infty$, $k(x)\in\LalocR$ for some $\alpha\in(1,\infty]$.
\vskip 2pt
\noindent{\bf(2-a).} When $p\leq s<r$, $\IR k^{\frac{s+\beta(s-q)}{s-r}}(x)h^{-\beta}(x)\dx<\infty$ for some $\beta\in\big[\frac{r}{q-r},\infty\big)$;
\vskip 2pt
\noindent{\bf(2-b).} when $1<p<N$ and $p<r<s\leq p^*$ or when $N\leq p<\infty$ and $p<r<s<\infty$,
\vskip 0pt\hspace{6.2mm} $\IR k^{\frac{s+\beta(s-q)}{s-r}}(x)h^{-\beta}(x)\dx<\infty$ for some $\beta\in\big[0,\frac{r}{q-r}\big]$.
\vskip 2pt
\noindent{\bf(3-a).} When $p<s<r$, $\IR k^{\frac{s+\beta(s-q)}{s-r}}(x)h^{-\beta}(x)\dx<\infty$ for some $\beta\in\big[\frac{s}{q-s},\frac{s(r-p)}{(s-p)(q-r)}\big)$;
\vskip 2pt
\noindent{\bf(3-b).} when $1<p<N$ and $p<r<s\leq p^*$ or when $N\leq p<\infty$ and $p<r<s<\infty$,
\vskip 0pt\hspace{6.2mm} $\IR k^{\frac{s+\beta(s-q)}{s-r}}(x)h^{-\beta}(x)\dx<\infty$ for some $\beta\in\big(\frac{s(r-p)}{(s-p)(q-r)},\frac{s}{q-s}\big]$.
\vskip 2pt
\noindent{\bf(4-a).} When $1<p<N$, $\bigintsss_\RN\big[\frac{k^{q-t}(x)}{h^{r-t}(x)}\big]^{\frac{\gamma}{q-r}}\dx<\infty$ for some $0\leq t<p$ and
$\gamma\in\big[\frac{p^*}{p^*-t},\frac{p}{p-t}\big]$,
\vskip 0pt\hspace{6.2mm} and $\bigintsss_\RN\big[\frac{k^{q-p}(x)}{h^{r-p}(x)}\big]^{\frac{\gamma}{q-r}}\dx<\infty$ for some $\gamma\in\big[\frac{p^*}{p^*-p},\infty\big]$ with $h(x),k(x)\in\LlocR$;
\vskip 2pt
\noindent{\bf(4-b).} when $N\leq p<\infty$, $\bigintsss_\RN\big[\frac{k^{q-t}(x)}{h^{r-t}(x)}\big]^{\frac{\gamma}{q-r}}\dx<\infty$ for some $0\leq t<p$ and
$\gamma\in\big(1,\frac{p}{p-t}\big]$\footnote{By abuse of notation, we include $t=0$ and $\gamma=1$ here.},
\vskip 0pt\hspace{6.2mm} and $\bigintsss_\RN\big[\frac{k^{q-p}(x)}{h^{r-p}(x)}\big]^{\frac{\gamma}{q-r}}\dx<\infty$ for some $\gamma\in(1,\infty]$ with $h(x),k(x)\in\LlocR$.
\vskip 2pt
\noindent{\bf(5-a).} When $1<p<N$, $k(x)\in\LsR$ for some $\sigma\in\big[\frac{p^*}{p^*-r},\infty\big]$;
\vskip 2pt
\noindent{\bf(5-b).} when $N\leq p<\infty$, $k(x)\in\LsR$ for some $\sigma\in(1,\infty]$.
\vskip 2pt

It is worthwhile to remark our hypotheses have all the known conditions regarding the integrability of the ratio function $k^\mu(x)/h^\nu(x)$ as special cases; for instance, $\beta=\frac{r}{q-r}$ or $t=0$ yields $\bigintsss_\RN\big[\frac{k^q(x)}{h^r(x)}\big]^{\frac{1}{q-r}}\dx<\infty$, $t=1$ and $\gamma=\frac{p^*}{p^*-1}$ for $1<p<N$ implies $\bigintsss_\RN\big[\frac{k^{q-1}(x)}{h^{r-1}(x)}\big]^{\frac{p^*}{(p^*-1)(q-r)}}\dx<\infty$, while $t=p$ and $\gamma=\frac{p^*}{p^*-p}$ for $1<p<N$ leads to $\bigintsss_\RN\big[\frac{k^{q-p}(x)}{h^{r-p}(x)}\big]^{\frac{p^*}{(p^*-p)(q-r)}}\dx<\infty$.

Notice hypotheses {\bf(1)-(2)} ensure the compact embedding $\E\hookrightarrow\LrkR$ (see Propositions \ref{P2.1} and \ref{P2.2} below).
In addition, if $\beta=0$ in {\bf(2-b)}, then one has assumption {\bf(5)} covered expect for $\sigma=\infty$, but condition {\bf(5)} is crucial in proving the existence of a second positive solution to \eqref{eq1.4}; so, it is better to keep it separately.
On the other hand, one can verify that the extremal cases in {\bf(2-b)} where $\beta=0$ and $\beta=\frac{r}{q-r}$ altogether lead to certain cases in assumption {\bf(4)} for appropriate indices as observed earlier; for example, the condition $\bigintsss_\RN\big[\frac{k^{q-t}(x)}{h^{r-t}(x)}\big]^{\frac{p^*}{(p^*-t)(q-r)}}\dx<\infty$ is weaker than $\bigintsss_\RN\big[\frac{k^q(x)}{h^r(x)}\big]^{\frac{1}{q-r}}\dx<\infty$ provided $r<p^*$ and $\IR k^{\frac{p^*}{p^*-r}}(x)\dx<\infty$.

Finally, for $\nm{u}^q_{\LqhRnm}=\IR h\n{u}^q\dx$ (and $\nm{u}^q_{\LrkRnm}=\IR k\n{u}^r\dx$ used later), define
\begin{equation}
\E=\bre{u\in\WpR:\nm{u}_\E:=\nm{u}_{\WpRnm}+\nm{u}_{\LqhRnm}<\infty}.\nonumber
\end{equation}
Then, our main result of this paper reads as below.

\begin{thm}\label{T1.1}
Under our {\bf Standing Assumptions (1)-(3)}, or {\bf(1), (2) and (4)}, there is a $\lambda_1\geq0$ such that equation \eqref{eq1.4} has at least a positive solution in $\E$ for every $\lambda>\lambda_1$.
Besides, when our {\bf Standing Assumption (5)} is also true, then $\lambda_1>0$ and equation \eqref{eq1.4} has at least a positive solution in $\E$ if and only if $\lambda\geq\lambda_1$ (except for $t=p$); moreover, there is a $\lambda_2(\geq\lambda_1)$ such that equation \eqref{eq1.4} has at least two positive solutions in $\E$ for every $\lambda>\lambda_2$.
\end{thm}

% +++ below is section 2 %%%%%%%%%%%%%%%%%%%%%%%%%%%%%%%%%%%%%%%%%%%%%%%%%%%%%%%%%%%%%%%%%%%%%%%%%%%%%%%%%%%%%%%%%%%%%%%%%%%%%%%%%%%%%%%%%%%%%%%%%%%
% &&& %%%%%%%%%%%%%%%%%%%%%%%%%%%%%%%%%%%%%%%%%%%%%%%%%%%%%%%%%%%%%%%%%%%%%%%%%%%%%%%%%%%%%%%%%%%%%%%%%%%%%%%%%%%%%%%%%%%%%%%%%%%%%%%%%%%%%%%%%%%%%%
\section{Proof of Theorem \ref{T1.1}}\label{PT1.1} % use lowercase except for proper names
\noindent In this section, we seek nontrivial positive solutions to \eqref{eq1.4} in $\E$ through identifying the critical points of the associated energy functional $\Jc:\E\to\R$, defined by
\begin{equation}\label{eq2.1}
\Jc(u)=\frac{1}{p}\IR\pr{\n{\gradu}^p+\n{u}^p}\dx+\frac{1}{q}\IR h\n{u}^q\dx-\frac{\lambda}{r}\IR k\pr{u^+}^r\dx.
\end{equation}

The first task here is to provide some compact embedding results of $\E\hookrightarrow\LrkR$.

\begin{prop}\label{P2.1}
Let $1\leq p<N$ and $p\leq s<r<\min\{p^*,q\}<\infty$, or let $N\leq p<\infty$ and $p\leq s<r<q<\infty$.
Suppose $h(x),k(x)>0$ satisfy $k(x)\in\LalocR$ for some $\alpha\in\big(\frac{p^*}{p^*-r},\infty\big]$ in the former case and $\alpha\in(1,\infty]$ in the latter case, while
$\IR k^{\frac{s+\beta(s-q)}{s-r}}(x)h^{-\beta}(x)\dx<\infty$ for some $\beta\in\big[\frac{r}{q-r},\infty\big)$.
Then, the embedding $\E\hookrightarrow\LrkR$ is compact.
\end{prop}

\begin{proof}
Recall $\E$ is a subspace of $\WpR$ by its definition.
Write $\mathfrak{x}=\frac{\beta(s-r)}{s+\beta(s-q)}$, $\mathfrak{y}=\frac{r+\beta(r-q)}{s+\beta(s-q)}$ and $\mathfrak{z}=\frac{s-r}{s+\beta(s-q)}$ for an arbitrarily chosen $s\in[p,r)$, and notice $\mathfrak{x}+\mathfrak{y}+\mathfrak{z}=1$.

Now, set $r_1=\mathfrak{x}^{-1}$, $r_2=\mathfrak{y}^{-1}$ and $r_3=\mathfrak{z}^{-1}$ to observe
\begin{equation}\label{eq2.2}
\begin{split}
&\IR k\n{u}^r\dx=\IR\bre{h\n{u}^q}^{\mathfrak{x}}\n{u}^{r-q\mathfrak{x}}\bre{k\hspace{0.2mm}h^{-\mathfrak{x}}}\dx\\
\leq&\Big(\IR h\n{u}^q\dx\Big)^{\frac{1}{r_1}}\Big(\IR\n{u}^s\dx\Big)^{\frac{1}{r_2}}\Big(\IR k^{r_3}h^{-\beta}\dx\Big)^{\frac{1}{r_3}}\\
\leq&\,C_1\Big(\IR h\n{u}^q\dx\Big)^{\frac{1}{r_1}}\Big(\IR\pr{\n{\gradu}^p+\n{u}^p}\dx\Big)^{\frac{s}{pr_2}}\Big(\IR k^{r_3}h^{-\beta}\dx\Big)^{\frac{1}{r_3}},
\end{split}
\end{equation}
with $s=r_2(r-q\mathfrak{x})$ for a constant $C_1>0$, provided $\mathfrak{x},\mathfrak{y},\mathfrak{z}>0$ simultaneously.

To see $\mathfrak{x}>0$ or $\mathfrak{z}>0$, we have $\beta>\frac{s}{q-s}$ and to see $\mathfrak{y}>0$, we have $\beta>\frac{r}{q-r}$, so that $\beta>\frac{r}{q-r}$ since $s<r$.
We certainly can take $\mathfrak{y}=0$ and consequently have $\beta=\frac{r}{q-r}$.

Next, one has $\frac{q}{r_1}+\frac{s}{r_2}=q\mathfrak{x}+s\mathfrak{y}=r$; so, the embedding $\E\to\LrkR$ is continuous.
Now, let $\bre{u_l:l\geq1}$ be a sequence of functions in $\E$, with $u_l\rightharpoonup0$ when $l\to\infty$ and $\nm{u_l}_\E$ uniformly bounded.
It follows that
\begin{equation}
\IR k\n{u_l}^r\dx=\IBR k\n{u_l}^r\dx+\IBRc k\n{u_l}^r\dx.\nonumber
\end{equation}
Here, and hereafter, $\BR$ denotes the ball of radius $R$ in $\RN$ that is centered at the origin and $\BRc=\RN\setminus\BR$.
For the integral over $\BRc$, we apply \eqref{eq2.2} on $\BRc$ to derive
\begin{equation}
\IBRc k\n{u_l}^r\dx\leq C_1\nm{k^{r_3}h^{-\beta}}^{\frac{1}{r_3}}_{1,\hspace{0.2mm}\BRc}\nm{u_l}^r_\E\to0\nonumber
\end{equation}
as $R\to\infty$.
For the integral over $\BR$, our (local) hypotheses lead to
\begin{equation}
W^{1,\hspace{0.2mm}p}\big(\BR\big)\hookrightarrow L^{\frac{\alpha r}{\alpha-1}}\big(\BR\big)\to L^r_k\big(\BR\big),\nonumber
\end{equation}
because either $r\leq\frac{\alpha r}{\alpha-1}<p^*$ when $1\leq p<N$ or $r\leq\frac{\alpha r}{\alpha-1}<\infty$ when $N\leq p<\infty$.
Hence, for a subsequence relabeled with the same index $l$, one proves $u_l\to0$ in $\LrkR$.
\end{proof}

\begin{prop}\label{P2.2}
Let $1\leq r<\min\{q,s\}<\infty$, with $p\leq s\leq p^*$ if $1\leq p<N$ and $p\leq s<\infty$ if $N\leq p<\infty$.
%Under the same all other assumptions in Proposition \ref{P2.1}, for some $\beta\in\big[0,\frac{r}{q-r}\big]$, the embedding $\E\hookrightarrow\LrkR$ is compact.
Suppose further $h(x),k(x)>0$ satisfy $k(x)\in\LalocR$ for some $\alpha\in\big(\frac{p^*}{p^*-r},\infty\big]$ in the former case and $\alpha\in(1,\infty]$ in the latter case, while $\IR k^{\frac{s+\beta(s-q)}{s-r}}(x)h^{-\beta}(x)\dx<\infty$ for some $\beta\in\big[0,\frac{r}{q-r}\big]$.
Then, the embedding $\E\hookrightarrow\LrkR$ is compact.
\end{prop}

\begin{proof}
The proof is almost identical to that of Proposition \ref{P2.1}, and the only difference lies in the verification of $\mathfrak{x},\mathfrak{y},\mathfrak{z}>0$.
To see $\mathfrak{x}>0$ or $\mathfrak{z}>0$, one has $\beta<\frac{s}{q-s}$ and to see $\mathfrak{y}>0$, one has $\beta<\frac{r}{q-r}$, so that $\beta<\frac{r}{q-r}$ since $s>r$.
Again, taking $\mathfrak{y}=0$ yields $\beta=\frac{r}{q-r}$.
Finally, for $\beta=0$, or equivalently for $\mathfrak{x}=0$, we work in the space $\WpR$ and have
\begin{equation}
\E\to\WpR\hookrightarrow\LrkR,\nonumber
\end{equation}
which is simply a special case of Proposition 4.2 in \cite{Ha1} and Theorem 4.6 in \cite{Ha3}.
\end{proof}

Next, we make elementary observations for all solutions to equation $\eqref{eq1.4}$.

\begin{lem}\label{L2.3}
Under the {\bf Standing Assumption (3)}, each solution $\ul\in\E$ to equation $\eqref{eq1.4}$ satisfies
\begin{equation}\label{eq2.3}
\IR\pr{\n{\gradul}^p+\n{\ul}^p}\dx+\IR h\n{\ul}^q\dx\leq\lambda^\rho C_{hk},
\end{equation}
while under the {\bf Standing Assumption (4)}, each solution $\ul\in\E$ to equation $\eqref{eq1.4}$ satisfies
\begin{equation}\label{eq2.4}
\Big(\IR\pr{\n{\gradul}^p+\n{\ul}^p}\dx\Big)^{\frac{p-t}{p}}\leq\lambda^\varrho C'_{hk}.
\end{equation}
Here, $\rho,\varrho>0$ and $C_{hk},C'_{hk}>0$ are absolute constants, independent of $\lambda,u$.
\end{lem}

\begin{proof}
First, one notices each solution $\ul$ to $\eqref{eq1.4}$ (not necessarily positive) satisfies
\begin{equation}\label{eq2.5}
\IR\pr{\n{\gradul}^p+\n{\ul}^p}\dx+\IR h\n{\ul}^q\dx=\lambda\IR k\n{\ul}^r\dx.
\end{equation}

For our assumption {\bf(3)}, denote $r_4=\frac{s+\beta(s-q)}{\beta(s-r)}=r_1$, $r_5=\frac{p}{s}r_2=\frac{p\{s+\beta(s-q)\}}{s\{r+\beta(r-q)\}}<r_2$ (as $p<s$) and $\rho=\frac{r_4r_5}{r_4r_5-r_4-r_5}>r_3>0$ in \eqref{eq2.2} of Proposition \ref{P2.1} to observe, via \eqref{eq2.5},
\begin{equation}\label{eq2.6}
\begin{split}
&\lambda\IR k\n{\ul}^r\dx\leq\Big(\IR\pr{\n{\gradul}^p+\n{\ul}^p}\dx\Big)^{\frac{1}{r_5}}\Big(\IR h\n{\ul}^q\dx\Big)^{\frac{1}{r_4}}\\
&\hspace{29.4mm}\times\bigg\{\lambda^\rho C^\rho_1\Big(\IR k^{r_3}h^{-\beta}\dx\Big)^{\frac{\rho}{r_3}}\bigg\}^{\frac{1}{\rho}}\\
\leq&\,\frac{1}{r_5}\IR\pr{\n{\gradul}^p+\n{\ul}^p}\dx+\frac{1}{r_4}\IR h\n{\ul}^q\dx+\lambda^\rho C_{hk}.
\end{split}
\end{equation}
So, \eqref{eq2.3} is verified by virtue of \eqref{eq2.6} if $\frac{1}{r_4}+\frac{1}{r_5}<1$, which is true provided $\beta\in\big[\frac{s}{q-s},\frac{s(r-p)}{(s-p)(q-r)}\big)$ when $s<r$ or $\beta\in\big(\frac{s(r-p)}{(s-p)(q-r)},\frac{s}{q-s}\big]$ when $r<s$.

For our assumption {\bf(4)}, rewrite \eqref{eq2.5} by switching $\IR h\n{\ul}^q\dx$ to the righthand side, and recall an elementary estimate $h_1\mathfrak{u}^\mu-h_2\mathfrak{u}^\nu\leq C_{\mu\nu}\big[\frac{h_1^\nu}{h_2^\mu}\big]^{\frac{1}{\nu-\mu}}$ for $h_1,h_2>0$, $\nu>\mu>0$, $\mathfrak{u}\geq0$ and $C_{\mu\nu}>0$.
Set $h_1=\lambda k$, $h_2=h$, $\mu=r-t$, $\nu=q-t$ and $\mathfrak{u}=\n{\ul}$ to arrive at
\begin{equation}\label{eq2.7}
\begin{split}
&\IR\pr{\n{\gradul}^p+\n{\ul}^p}\dx\leq\lambda^{\frac{q-t}{q-r}}C(q,r,t)\IR\Big[\frac{k^{q-t}(x)}{h^{r-t}(x)}\Big]^{\frac{1}{q-r}}\n{\ul}^t\dx\\
\leq&\,\lambda^{\frac{q-t}{q-r}}C(q,r,t)\Big(\IR\Big[\frac{k^{q-t}(x)}{h^{r-t}(x)}\Big]^{\frac{\gamma}{q-r}}\dx\Big)^{\frac{1}{\gamma}}
\Big(\IR\n{\ul}^{\frac{\gamma t}{\gamma-1}}\dx\Big)^{\frac{\gamma-1}{\gamma}}\\
\leq&\,\lambda^{\frac{q-t}{q-r}}C'_{hk}\Big(\IR\pr{\n{\gradul}^p+\n{\ul}^p}\dx\Big)^{\frac{t}{p}}
\end{split}
\end{equation}
for $C(q,r,t)>0$; thus, \eqref{eq2.4} is a direct consequence of \eqref{eq2.7} with $\varrho=\frac{q-t}{q-r}>0$.
\end{proof}

\begin{lem}\label{L2.4}
Under the {\bf Standing Assumptions (1)-(3)} or {\bf(1)-(2)\&(4)}, $\Jc$ is of class $C^1$, and is coercive and sequentially weakly lower semicontinuous in $\E$.
\end{lem}

\begin{proof}
The ideas follow from \cite[Lemma 3.2]{Ha4}; so, we only stress the differences.

For our assumption {\bf(3)}, similar to \eqref{eq2.6}, one has
\begin{equation}
\begin{split}
&\frac{\lambda}{r}\IR k\pr{u^+}^r\dx\leq\frac{\lambda}{r}\IR k\n{u}^r\dx\leq\Big(\frac{r_5}{2p}\IR\pr{\n{\gradu}^p+\n{u}^p}\dx\Big)^{\frac{1}{r_5}}\\
&\hspace{2.7mm}\times\Big(\frac{r_4}{2q}\IR h\n{u}^q\dx\Big)^{\frac{1}{r_4}}
\bigg\{\lambda^\rho\Big(\frac{C_1}{r}\Big)^\rho\Big(\frac{2p}{r_5}\Big)^{\frac{\rho}{r_5}}\Big(\frac{2q}{r_4}\Big)^{\frac{\rho}{r_4}}
\Big(\IR k^{r_3}h^{-\beta}\dx\Big)^{\frac{\rho}{r_3}}\bigg\}^{\frac{1}{\rho}}\\
\leq&\,\frac{1}{2p}\IR\pr{\n{\gradu}^p+\n{u}^p}\dx+\frac{1}{2q}\IR h\n{u}^q\dx+\lambda^\rho\widetilde{C}_{hk},\nonumber
\end{split}
\end{equation}
so that
\begin{equation}\label{eq2.8}
\Jc(u)\geq\frac{1}{2p}\IR\pr{\n{\gradu}^p+\n{u}^p}\dx+\frac{1}{2q}\IR h\n{u}^q\dx-\lambda^\rho\widetilde{C}_{hk}.
\end{equation}

For our assumption {\bf(4)}, when $0\leq t<p$, similar to \eqref{eq2.7}, one has
\begin{equation}
\begin{split}
&\frac{\lambda}{r}\IR k\pr{u^+}^r\dx-\frac{1}{2q}\IR h\n{u}^q\dx\leq\frac{\lambda}{r}\IR k\n{u}^r\dx-\frac{1}{2q}\IR h\n{u}^q\dx\\
\leq&\,C(q,r,t,\lambda)\Big(\IR\Big[\frac{k^{q-t}(x)}{h^{r-t}(x)}\Big]^{\frac{\gamma}{q-r}}\dx\Big)^{\frac{1}{\gamma}}
\Big(\IR\n{u}^{\frac{\gamma t}{\gamma-1}}\dx\Big)^{\frac{\gamma-1}{\gamma}}\\
\leq&\,C_{hk}(q,r,t,\lambda)\Big(\IR\pr{\n{\gradu}^p+\n{u}^p}\dx\Big)^{\frac{t}{p}}\leq\frac{1}{2p}\IR\pr{\n{\gradu}^p+\n{u}^p}\dx+C_{hk\lambda}\nonumber
\end{split}
\end{equation}
for suitable constants $C(q,r,t,\lambda),C_{hk}(q,r,t,\lambda),C_{hk\lambda}>0$, which further implies
\begin{equation}
\frac{\lambda}{r}\IR k\pr{u^+}^r\dx\leq\frac{1}{2p}\IR\pr{\n{\gradu}^p+\n{u}^p}\dx+\frac{1}{2q}\IR h\n{u}^q\dx+C_{hk\lambda},\nonumber
\end{equation}
so that \eqref{eq2.8} holds after $\lambda^\rho\widetilde{C}_{hk}$ being replaced by $C_{hk\lambda}$.

On the other hand, when $t=p$ in our assumption {\bf(4)}, one has
\begin{equation}\label{eq2.9}
\begin{split}
&\frac{\lambda}{r}\IBRc k\pr{u^+}^r\dx-\frac{1}{2q}\IBRc h\n{u}^q\dx\\%\leq\frac{\lambda}{r}\IBRc k\n{u}^r\dx-\frac{1}{2q}\IBRc h\n{u}^q\dx\\
\leq&\,C(q,r,t,\lambda)\Big(\IBRc\Big[\frac{k^{q-p}(x)}{h^{r-p}(x)}\Big]^{\frac{\gamma}{q-r}}\dx\Big)^{\frac{1}{\gamma}}
\Big(\IBRc\n{u}^{\frac{\gamma p}{\gamma-1}}\dx\Big)^{\frac{\gamma-1}{\gamma}}\\
\leq&\Big(\frac{1}{2p}-\varepsilon\Big)\IBRc\pr{\n{\gradu}^p+\n{u}^p}\dx\leq\Big(\frac{1}{2p}-\varepsilon\Big)\IR\pr{\n{\gradu}^p+\n{u}^p}\dx
\end{split}
\end{equation}
for a sufficiently large $R>0$ and some $\varepsilon$ with $0<\varepsilon<\frac{1}{2p}$, while
\begin{equation}\label{eq2.10}
\begin{split}
&\frac{\lambda}{r}\IBR k\pr{u^+}^r\dx-\frac{1}{2q}\IBR h\n{u}^q\dx\\%\leq\frac{\lambda}{r}\IBR k\n{u}^r\dx-\frac{1}{2q}\IBR h\n{u}^q\dx\\
\leq&\,\varepsilon\IBR\pr{\n{\gradu}^p+\n{u}^p}\dx+\widetilde{C}_{hk\lambda}\leq\varepsilon\IR\pr{\n{\gradu}^p+\n{u}^p}\dx+\widetilde{C}_{hk\lambda}
\end{split}
\end{equation}
follows from the analysis of Lemma 2.4 in \cite{AT} seeing $h(x),k(x)\in\LlocR$ for another suitable constant $\widetilde{C}_{hk\lambda}>0$, so that we have \eqref{eq2.8}, with $\widetilde{C}_{hk\lambda}$ replacing $\lambda^\rho\widetilde{C}_{hk}$, via \eqref{eq2.9} and \eqref{eq2.10} (and the latter is derived through the subsequent estimates \eqref{eq2.11} and \eqref{eq2.12}).

Here, $\widetilde{C}_{hk},C_{hk\lambda},\widetilde{C}_{hk\lambda}>0$ are absolute constants, independent of $u$.

As a matter of fact, when $t=p$ in our assumption {\bf(4)}, given $M,\delta>0$, set
\begin{equation}
\left\{\begin{array}{ll}
\mathbf{X}=\bre{x\in\BR:k(x)<M\,\,\text{and}\,\,h(x)>\delta},\\
\mathbf{Y}=\bre{x\in\BR:k(x)<M\,\,\text{and}\,\,h(x)\leq\delta}\\
\hspace{4.1mm}\varsubsetneq\mathbf{Y}'=\bre{x\in\BR:h(x)\leq\delta},\\
\mathbf{Z}=\bre{x\in\BR:k(x)\geq M}.
\end{array}\right.\nonumber
\end{equation}
Then, $\mathscr{L}(\mathbf{Z})\to0$ as $M\to\infty$ and $\mathscr{L}(\mathbf{Y}')\to0$ as $\delta\to0$, so that
\begin{equation}\label{eq2.11}
\begin{split}
&\frac{\lambda}{r}\int_{\mathbf{Y}\cup\mathbf{Z}}k\pr{u^+}^r\dx-\frac{1}{2q}\int_{\mathbf{Y}\cup\mathbf{Z}}h\n{u}^q\dx\\
\leq&\,C(q,r,t,\lambda)\Big(\int_{\mathbf{Y}\cup\mathbf{Z}}\Big[\frac{k^{q-p}(x)}{h^{r-p}(x)}\Big]^{\frac{\gamma}{q-r}}\dx\Big)^{\frac{1}{\gamma}}
\Big(\int_{\mathbf{Y}\cup\mathbf{Z}}\n{u}^{\frac{\gamma p}{\gamma-1}}\dx\Big)^{\frac{\gamma-1}{\gamma}}\\
\leq&\,C(q,r,t,\lambda)\Big(\int_{\mathbf{Y}'\cup\mathbf{Z}}\Big[\frac{k^{q-p}(x)}{h^{r-p}(x)}\Big]^{\frac{\gamma}{q-r}}\dx\Big)^{\frac{1}{\gamma}}
\Big(\int_{\mathbf{Y}'\cup\mathbf{Z}}\n{u}^{\frac{\gamma p}{\gamma-1}}\dx\Big)^{\frac{\gamma-1}{\gamma}}\\
\leq&\,\frac{\varepsilon}{2}\int_{\mathbf{Y}'\cup\mathbf{Z}}\pr{\n{\gradu}^p+\n{u}^p}\dx\leq\frac{\varepsilon}{2}\IR\pr{\n{\gradu}^p+\n{u}^p}\dx
\end{split}
\end{equation}
provided $M$ is sufficiently large and $\delta$ is sufficiently small.
Next, for each $\tilde{t}$ with $0\leq\tilde{t}<p$ and some suitable constants $C_{M\delta}(q,r,t,\lambda),\widetilde{C}_{hk\lambda}>0$, one instead derives that
\begin{equation}\label{eq2.12}
\begin{split}
&\frac{\lambda}{r}\int_{\mathbf{X}}k\pr{u^+}^r\dx-\frac{1}{2q}\int_{\mathbf{X}}h\n{u}^q\dx\\
\leq&\,C(q,r,t,\lambda)\Big(\int_{\BR}\Big[\frac{k^{q-\tilde{t}}(x)}{h^{r-\tilde{t}}(x)}\Big]^{\frac{\gamma}{q-r}}\dx\Big)^{\frac{1}{\gamma}}
\Big(\int_{\BR}\n{u}^{\frac{\gamma\tilde{t}}{\gamma-1}}\dx\Big)^{\frac{\gamma-1}{\gamma}}\\
\leq&\,C_{M\delta}(q,r,t,\lambda)\Big(\IBR\pr{\n{\gradu}^p+\n{u}^p}\dx\Big)^{\frac{\tilde{t}}{p}}\leq\frac{\varepsilon}{2}\IR\pr{\n{\gradu}^p+\n{u}^p}\dx+\widetilde{C}_{hk\lambda}.
\end{split}
\end{equation}

Since $\Jc$ is coercive and $\E$ is reflexive (see \cite[Proposition A.11]{AP1}), each sequence $\bre{u_l:l\geq1}$ in $\E$ with $\Jc(u_l)$ bounded has a subsequence, using again $\bre{u_l:l\geq1}$, such that $u_l\rightharpoonup u\in\E$.
The compact embedding $\E\hookrightarrow\LrkR$, via assumptions {\bf(1)-(2)}, leads to
\begin{equation}\label{eq2.13}
\lim_{l\to\infty}\IR k\pr{u^+_l}^r\dx=\IR k\pr{u^+}^r\dx
\end{equation}
for yet another subsequence, which along with the lower semicontinuity of norms in $\WpR$ and $\LqhR$ generates the sequentially weak lower semicontinuity of $\Jc$ in $\E$.
\end{proof}

Define
\begin{equation}\label{eq2.14}
\tilde{\lambda}=\inf_{u\in\E,\nm{u^+}_{\LrkRnm}=1}\bre{\frac{r}{p}\IR\pr{\n{\gradu}^p+\n{u}^p}\dx+\frac{r}{q}\IR h\n{u}^q\dx};
\end{equation}
then, $\tilde{\lambda}>0$.
In fact, if not, then there is a sequence $\bre{u_l:l\geq1}$ in $\E$ such that $\nm{u^+_l}_{\LrkRnm}=1$ but
$\frac{r}{p}\IR\pr{\n{\nabla u_l}^p+\n{u_l}^p}\dx+\frac{r}{q}\IR h\n{u_l}^q\dx\to0$; this would yield $\nm{u_l}_\E\to0$, contradicting the compact embedding $\E\hookrightarrow\LrkR$ in view of $\nm{u_l}_{\LrkRnm}\geq\nm{u^+_l}_{\LrkRnm}=1$.

Denote by $\lambda^*$ the supermum of $\lambda$ such that \eqref{eq1.4} has no nontrivial positive solution for each $\mu<\lambda$, and denote by $\lambda^{**}$ the infimum of $\lambda$ such that \eqref{eq1.4} has at least one nontrivial positive solution at $\lambda$.
Then, we have $0\leq\lambda^*=\lambda^{**}\leq\tilde{\lambda}$.
Actually, for each $\lambda>\tilde{\lambda}$,
\begin{equation}
\lambda\IR k\pr{\vl^+}^r\dx>\frac{r}{p}\IR\pr{\n{\gradvl}^p+\n{\vl}^p}\dx+\frac{r}{q}\IR h\n{\vl}^q\dx\nonumber
\end{equation}
follows, with some $\vl\in\E$, through homogeneity, which can be rewritten as
\begin{equation}
\Jc(\vl)=\frac{1}{p}\IR\pr{\n{\gradvl}^p+\n{\vl}^p}\dx+\frac{1}{q}\IR h\n{\vl}^q\dx-\frac{\lambda}{r}\IR k\pr{\vl^+}^r\dx<0.\nonumber
\end{equation}
This, together with Lemma \ref{L2.4}, leads to $\Jc(\ul)=\inf\limits_{u\in\E}\Jc(u)\leq\Jc(\vl)<0$ for an $\ul\geq0$ in $\E$ since $\Jc(\n{\ul})\leq\Jc(\ul)$, which is a nontrivial positive solution to equation \eqref{eq1.4}.
So, one has $\lambda^{**}\leq\tilde{\lambda}$.
On the other hand, if $\lambda^*>\lambda^{**}$, one would find a $\lambda'\in[\lambda^{**},\lambda^*)$ such that equation \eqref{eq1.4} has at least one nontrivial positive solution at $\lambda'$ according to the definition of $\lambda^{**}$ - this however is against the definition of $\lambda^*$; if $\lambda^*<\lambda^{**}$, one would find a $\lambda'\in(\lambda^*,\lambda^{**}]$ and then a $\mu'(<\lambda')$ such that equation \eqref{eq1.4} has at least one nontrivial positive solution at $\mu'$ according to the definition of $\lambda^*$ - this however is against the definition of $\lambda^{**}$.
So, $\lambda^*=\lambda^{**}$.

Write $\lambda_1:=\lambda^*=\lambda^{**}$ in the sequel.

\begin{prop}\label{P2.5}
Under the {\bf Standing Assumptions (1)-(3)} or {\bf(1)-(2)\&(4)}, $\lambda_1\geq0$ and equation \eqref{eq1.4} has a nontrivial positive solution $\ul\geq0$ in $\E$ for every $\lambda>\lambda_1$.
\end{prop}

\begin{proof}
By definition, $\lambda_1=\lambda^*$; so, if $\ul\geq0$ is a nontrivial positive solution to \eqref{eq1.4} in $\E$, then $\lambda\geq\lambda_1$.
Now, we verify \eqref{eq1.4} has at least a nontrivial solution $\ul\geq0$ in $\E$ for each $\lambda>\lambda_1$ via Struwe \cite[Theorem 2.4]{St}; see also
\cite[Theorem 4.2]{AP1} and \cite[Proposition 3.3]{Ha4}.

By definition, $\lambda_1=\lambda^{**}$; so, one finds a $\mu\in[\lambda_1,\lambda)$ at which \eqref{eq1.4} has a nontrivial solution $\um\geq0$ in $\E$ that is a subsolution to \eqref{eq1.4} at $\lambda$.
Consider the constrained minimizing problem $\inf\limits_{u\in\Ms}\Jc(u)$ for $\Ms=\bre{u\in\E:u\geq\um\geq0}$.
Since $\Ms$ is closed and convex, it is weakly closed in $\E$.
Thus, Lemma \ref{L2.4} yields an $\ul(\geq\um)$ in $\Ms$ with $\Jc(\ul)=\inf\limits_{u\in\Ms}\Jc(u)$.
Take $\varphi\in\CCON$; set $\vv=\max\{0,\um-\ul+\varepsilon\varphi\}\geq0$ and $\ve=\vv+\ul-\varepsilon\varphi\hspace{0.2mm}(\geq\um)$ in $\Ms$ for some $\varepsilon>0$.
Then, we have $\Jc'(\ul)(\ul)\leq\Jc'(\ul)(\ve)$, which further leads to
\begin{equation}\label{eq2.15}
\Jc'(\ul)(\varphi)\leq\frac{1}{\varepsilon}\Jc'(\ul)(\vv).
\end{equation}

Put $\Ov=\bre{x\in\RN:\vv(x)>0}=\bre{x\in\RN:\ul(x)-\um(x)<\varepsilon\varphi(x)}\varsubsetneq\operatorname{supp}(\varphi^+)$.
As $\um$ is a subsolution to \eqref{eq1.4} at $\lambda$ and $\vv\geq0$, one has $\Jc'(\um)(\vv)\leq0$ so that
\begin{equation}
\begin{split}
&\,\Jc'(\ul)(\vv)\leq\Jc'(\ul)(\vv)-\Jc'(\um)(\vv)\\
\leq&-\IOv\big(\n{\gradul}^{p-2}\gradul-\n{\gradum}^{p-2}\gradum\big)\cdot\pr{\gradul-\gradum}\dx\\
&+\varepsilon\bigg\{\IOv\big(\n{\gradul}^{p-2}\gradul-\n{\gradum}^{p-2}\gradum\big)\cdot\nabla\varphi\,\dx+\IOv\big(\ul^{p-1}-\um^{p-1}\big)\n{\varphi}\dx\\
&\hspace{8mm}+\IOv h\big(\ul^{q-1}-\um^{q-1}\big)\n{\varphi}\dx+\lambda\IOv k\big(\ul^{r-1}-\um^{r-1}\big)\n{\varphi}\dx\bigg\}\\
\leq&\,\varepsilon\Big\{\nm{\varphi}_{\WpOvnm}\big(\nm{\ul}^{p-1}_{\WpOvnm}+\nm{\um}^{p-1}_{\WpOvnm}\big)\\
&\hspace{4mm}+\nm{\varphi}_{\LqhOvnm}\big(\nm{\ul}^{q-1}_{\LqhOvnm}+\nm{\um}^{q-1}_{\LqhOvnm}\big)\\
&\hspace{4mm}+\lambda\nm{\varphi}_{\LrkOvnm}\big(\nm{\ul}^{r-1}_{\LrkOvnm}+\nm{\um}^{r-1}_{\LrkOvnm}\big)\Big\}=o(\varepsilon)\nonumber
\end{split}
\end{equation}
as $\varepsilon\to0^+$, seeing $0<\vv\leq\varepsilon\n{\varphi}$ on $\Ov$.
This combined with \eqref{eq2.15} yields $\Jc'(\ul)(\varphi)\leq0$ for every $\varphi\in\CCON$; so, it further implies $\Jc'(\ul)(-\varphi)\leq0$.
Via density, $\Jc'(\ul)(v)=0$ for all $v\in\E$.
Therefore, $\ul(\geq\um\geq0)$ is a nontrivial solution to equation \eqref{eq1.4} at $\lambda$.
\end{proof}

Proposition \ref{P2.5} obviously provides the proof to the first assertion of Theorem \ref{T1.1}, and below we shall discuss the second twofold assertion of Theorem \ref{T1.1}.

\begin{lem}\label{L2.6}
Under the {\bf Standing Assumption (5)}, there exists an absolute constant $C_k>0$ such that each nontrivial solution $\ul\in\E$ to equation $\eqref{eq1.4}$ satisfies
\begin{equation}\label{eq2.16}
\IR k\n{\ul}^r\dx\geq(\lambda\hspace{0.2mm}C_k)^{\frac{r}{p-r}}.
\end{equation}
Besides, when the {\bf Standing Assumption (3)} or {\bf(4)} also holds, then there is another absolute constant $\widehat{C}_{hk}>0$ such that $\lambda_1\geq\widehat{C}_{hk}>0$.
\end{lem}

\begin{proof}
First, when our assumption {\bf(5)} is true, one observes
\begin{equation}\label{eq2.17}
\begin{split}
&\Big(\IR k\n{u}^r\dx\Big)^{\frac{p}{r}}
\leq\bigg\{\Big(\IR k^\sigma\dx\Big)^{\frac{1}{\sigma}}\Big(\IR\n{u}^{\frac{\sigma r}{\sigma-1}}\dx\Big)^{\frac{\sigma-1}{\sigma}}\bigg\}^{^{\frac{p}{r}}}\\
\leq&\,C_k\IR\pr{\n{\gradu}^p+\n{u}^p}\dx.
\end{split}
\end{equation}
Combining \eqref{eq2.17} with \eqref{eq2.5}, one has $\big(\IR k\n{\ul}^r\dx\big)^{\frac{p}{r}}\leq\lambda\hspace{0.2mm}C_k\IR k\n{\ul}^r\dx$ for all nontrivial solutions $\ul$ to \eqref{eq1.4} (not necessarily positive), which in turn yields \eqref{eq2.16}.

Now, when our assumption {\bf(3)} also holds, by \eqref{eq2.3}, \eqref{eq2.16} and \eqref{eq2.17}, $(\lambda\hspace{0.2mm}C_k)^{\frac{p}{p-r}}\leq C_k\lambda^\rho C_{hk}$ follows so that $\lambda\geq\breve{C}_{hk}>0$ since $p<r$, which in particular yields $\lambda_1\geq\breve{C}_{hk}>0$.
On the other hand, when our assumption {\bf(4)} also holds for $t=p$, \eqref{eq2.4} automatically implies $\lambda_1\geq\breve{C}'_{hk}>0$; if $0\leq t<p$ in {\bf(4)}, then by \eqref{eq2.4}, \eqref{eq2.16} and \eqref{eq2.17}, $(\lambda\hspace{0.2mm}C_k)^{\frac{p}{p-r}}\leq C_k(\lambda^\varrho C'_{hk})^{\frac{p}{p-t}}$ follows so that likewise $\lambda_1\geq\breve{C}'_{hk}>0$.
Thus, the proof is done for $\widehat{C}_{hk}=\min\big\{\breve{C}_{hk},\breve{C}'_{hk}\big\}>0$.
\end{proof}

\begin{prop}\label{P2.7}
Under the {\bf Standing Assumptions (1)-(3)\&(5)} or {\bf(1)-(2)\&(4)-(5)}, we have $\lambda_1>0$ and equation \eqref{eq1.4} has a nontrivial positive solution $\ul$ in $\E$ if and only if $\lambda\geq\lambda_1$ (except for $t=p$); moreover, for $\lambda_2:=\tilde{\lambda}\hspace{0.2mm}(\geq\lambda_1)$ as introduced in \eqref{eq2.14}, equation \eqref{eq1.4} has at least two distinct nontrivial positive solutions $\ul\neq\ult$ in $\E$ for every $\lambda>\lambda_2$.
\end{prop}

\begin{proof}
Recall $\lambda_1>0$ has been proven.
So, we only need to show \eqref{eq1.4} has a nontrivial positive solution at $\lambda_1$.
Assume $\bre{\lambda(l)>\lambda_1:l\geq1}$ decrease to $\lambda_1$, with $\bre{\ull\geq0:l\geq1}$ an associated sequence of nontrivial solutions to \eqref{eq1.4}, which is bounded in $\E$ by Lemma \ref{L2.3}.

Actually, \eqref{eq2.3} provides the boundedness of $\bre{\ull\geq0:l\geq1}$ in $\E$ immediately while \eqref{eq2.4} implies that of $\bre{\ull\geq0:l\geq1}$ in $\WpR$ for
$0\leq t<p$ in {\bf(4)}; rewrite \eqref{eq2.5} by switching only one half of $\IR h\n{\ul}^q\dx$ to the righthand side (but keeping the other half) and reuse the analysis of \eqref{eq2.7} to observe the boundedness of $\bre{\ull\geq0:l\geq1}$ in $\LqhR$.

Thus, there exists a subsequence $\bre{\ull:l\geq1}$, using the same index $\lambda(l)$, with $\ull\rightharpoonup\omega$ in $\E$ and $\ull\to\omega$ in $\LrkR$ for some function $\omega\in\E$ such that $\ull\to\omega\geq0$ \textsl{a.e.} on $\RN$.
Then, one has \eqref{eq2.13} for $\ull,\omega\geq0$ (instead of $u_l,u$) and
\begin{equation}
\lim_{l\to\infty}\IR k\ull^{r-1}v\dx=\IR k\omega^{r-1}v\dx,\hspace{6mm}\forall\hspace{2mm}v\in\E;\nonumber
\end{equation}
see for example \cite[Lemma 3.4]{AP1}.
Keep in mind $\Jcll'(\ull)=0$; that is,
\begin{equation}\label{eq2.18}
\begin{split}
&\IR\n{\gradull}^{p-2}\gradull\cdot\gradv\dx+\IR\ull^{p-1}v\dx+\IR h\ull^{q-1}v\dx\\
=&\,\lambda(l)\IR k\ull^{r-1}v\dx,\hspace{6mm}\forall\hspace{2mm}v\in\E.
\end{split}
\end{equation}

As $\Jc'(\omega)$ is a continuous, linear functional in $\E$,
\begin{equation}\label{eq2.19}
\begin{split}
0\leftarrow&\,\Jcll'(\ull)(\ull-\omega)+\lambda(l)\IR k\ull^{r-1}(\ull-\omega)\dx\\
&-\Jc'(\omega)(\ull-\omega)-\lambda\IR k\omega^{r-1}(\ull-\omega)\dx\\
=&\IR\big(\n{\gradull}^{p-2}\gradull-\n{\gradw}^{p-2}\gradw\big)\cdot(\gradull-\gradw)\dx\\
&+\IR\big(\ull^{p-1}-\omega^{p-1}\big)(\ull-\omega)\dx+\IR h\big(\ull^{q-1}-\omega^{q-1}\big)(\ull-\omega)\dx
\end{split}
\end{equation}
follows when $l\to\infty$, from which, together with \cite[Lemma 3.2]{Ha2}, one concludes $\ull\to\omega$ in $\E$.
This particularly implies $\gradull\to\gradw$ \textsl{a.e.} on $\RN$, so that we also have
\begin{equation}
\begin{split}
&\lim_{l\to\infty}\IR\n{\gradull}^{p-2}\gradull\cdot\gradv\dx=\IR\n{\gradw}^{p-2}\gradw\cdot\gradv\dx,\\
&\lim_{l\to\infty}\IR\ull^{p-1}v\dx=\IR\omega^{p-1}v\dx,\\
&\lim_{l\to\infty}\IR h\ull^{q-1}v\dx=\IR h\omega^{q-1}v\dx,\hspace{6mm}\forall\hspace{2mm}v\in\E.\nonumber
\end{split}
\end{equation}

Recall $\omega\geq0$.
Upon letting $l\to\infty$ on both sides of \eqref{eq2.18}, for all $v\in\E$, we have
\begin{equation}\label{eq2.20}
\IR\n{\gradw}^{p-2}\gradw\cdot\gradv\dx+\IR\omega^{p-1}v\dx+\IR h\omega^{q-1}v\dx=\lambda_1\IR k\omega^{r-1}v\dx.
\end{equation}
Because \eqref{eq2.16} yields $\IR k\omega^r\dx\geq(\lambda(1)\hspace{0.2mm}C_k)^{\frac{r}{p-r}}>0$ in view of the compact embedding $\E\hookrightarrow\LrkR$, $u_{\lambda_1}:=\omega\geq0$ is a nontrivial positive solution to equation \eqref{eq1.4} in $\E$.

On the other hand, from the discussions in \cite[Lemma 3]{PR} or \cite[Lemma 3]{Ra}, one knows any other solution $\ult\geq0$ to \eqref{eq1.4} at $\lambda\hspace{0.2mm}(>\lambda_2)$, if it exists, should satisfy $\ult\leq\ul$ with $\Jc(\ul)<0$.
Furthermore, it is easily seen from \eqref{eq2.17} that
\begin{equation}\label{eq2.21}
\begin{split}
&\Jc(u)\geq\frac{1}{p}\IR\pr{\n{\gradu}^p+\n{u}^p}\dx-\frac{\lambda}{r}\Big(C_k\IR\pr{\n{\gradu}^p+\n{u}^p}\dx\Big)^{\frac{r}{p}}\\
\geq&\bigg\{\frac{1}{p}-\frac{\lambda}{r}C^{\frac{r}{p}}_k\Big(\IR\pr{\n{\gradu}^p+\n{u}^p}\dx\Big)^{\frac{r-p}{p}}\bigg\}\IR\pr{\n{\gradu}^p+\n{u}^p}\dx\geq\eta>0,
\end{split}
\end{equation}
provided $0<\nm{u}_{\WpRnm}<\min\Big\{\nm{\ul}_{\WpRnm},\big(\frac{r}{\lambda p}C^{-\frac{r}{p}}_k\big)^{\frac{1}{r-p}}\Big\}$.
Observe that the mountain pass structure by \eqref{eq2.21} is on $\WpR$ rather than on $\E$.
Thus, applying the mountain pass theorem of Candela and Palmieri \cite[Theorem 2.5]{CP} (see \cite[Theorem A.3]{AP1} as well for a different proof), there exists a sequence of functions $\bre{\tilde{u}_l:l\geq1}$ in $\E$ that satisfy $\Jc(\tilde{u}_l)\to c>0$ and $\Jc'(\tilde{u}_l)\to0$ in the dual space of $\E$ when $l\to\infty$, where we set $c=\inf\limits_{h\in\mathscr{H}}\max\limits_{z\in[0,1]}\Jc(h(z))>0$ for $\mathscr{H}=\bre{h\in C\big([0,1];\E\big):h(0)=0,\,h(1)=\ul}$.
As a result, we find a subsequence $\bre{\tilde{u}_l:l\geq1}$, using the same notation, such that $\tilde{u}_l\rightharpoonup\xi$ in $\E$ and $\tilde{u}_l\to\xi$ in $\LrkR$ for some function $\xi\in\E$ with $\tilde{u}_l\to\xi$ \textsl{a.e.} on $\RN$.
The same discussion as done ahead of \eqref{eq2.20} then leads to $\tilde{u}_l\to\xi$ in $\E$.
In particular, we have $\Jc'(\xi)v=0$ for all $v\in\E$; upon using $v=\xi^-\in\E$ as a test function, one realizes $\xi^-\equiv0$ and $\xi\geq0$.
As a consequence, $\ult:=\xi\geq0$ is another nontrivial positive solution to equation \eqref{eq1.4} in $\E$ that is different from $\ul$, since $\Jc(\ult)=c>0$.
\end{proof}

% -----------------------------------------------------------------------------

\bibliographystyle{amsplain}
%\bibliography{xbib}

\end{document}